\documentclass[12pt,letterpaper,reqno]{amsart}
\usepackage{hyperref}

\usepackage{amsthm}
\usepackage{amscd}
\usepackage{amsmath}
\usepackage{graphicx}
\usepackage{color}
\usepackage{amsfonts}

\usepackage{enumerate}

\usepackage[T1]{fontenc}
\usepackage{textcomp}
\usepackage{lmodern}

\usepackage{microtype}

\usepackage[text={6in, 7.8in}, centering]{geometry}

\usepackage{parskip}
 
\newtheorem{thm}{Theorem}[section]

\newtheorem{lem}[thm]{Lemma}

\newtheorem{prop}[thm]{Proposition}
\newtheorem{cor}[thm]{Corollary}
\newtheorem*{rmk}{Remark}

\newcommand{\Z}{\mathbb{Z}}

\newcommand{\R}{\mathbb{R}}

\newcommand{\cA}{\mathcal{A}}
\newcommand{\C}{\mathbb{C}}
\newcommand{\Q}{\mathbb{Q}}
\newcommand{\cD}{\mathcal{D}}
\newcommand{\cN}{\mathcal{N}}
\newcommand{\cM}{\mathcal{M}}

\newcommand{\T}{\mathbb{T}}
\newcommand{\Span}{\mathrm{Span}\,}

\newcommand{\pos}{P\"{o}schel}

\newcommand{\floor}[1]{\lfloor #1 \rfloor}

\newcommand{\dist}{\mathrm{dist}\,}

\newcommand{\bmat}[1]{\begin{bmatrix}#1\end{bmatrix}}

\author{Jianlu Zhang\dag}
\address{\dag\  Institute of Theoretical Studies, ETH Z\"urich, CH-8092 Z\"urich, Switzerland }
\email{jianlu.zhang@math.ethz.ch}
\author{Ke Zhang\ddag}
\address{\ddag\ Department of Mathematics, University of Toronto, Toronto, Canada}
\email{kzhang@math.toronto.edu}


\title[Global stability of nearly integrable systems]{Improved stability for analytic quasi-convex  nearly integrable systems and optimal speed of Arnold diffusion}

\begin{document}
\maketitle

\begin{abstract}
We improve the global Nekhoroshev stability  for  analytic quasi-convex nearly integrable Hamiltonian systems. The new stability result is optimal, as it matches the fastest speed of Arnold diffusion. 
\end{abstract}

\section{Introduction}

We consider a real analytic Hamiltonian 
\[
	H(\theta, I) = h(I) + f(\theta, I), \quad 
	I \in \R^n, \quad \theta \in \T^n = (\R/2\pi\Z)^n,
\]
with $|f| < \epsilon \ll 1$. It is a classical result of Nekhoroshev (\cite{Nek77}, \cite{Nek79}) that when $h(I)$ satisfies a non-degeneracy condition known as steepness (see also the modern treatments of \cite{Nie07}, \cite{BN12}), the system enjoys a global stretched exponential stability, of the type 
\[
	\|I(t) - I(0)\| \le C \epsilon^b, \quad \text{ for } |t| \le 
	\exp\left( - C^{-1}  \epsilon^{-a} \right).
\]
In the case when the integrable Hamiltonian is quasi-convex (see definition below), the system enjoys the largest stability exponent $b$.  Lochak and Neishtadt,  also P\"{o}schel (see for example \cite{LN92}, \cite{Loc92}, \cite{Pos93}) obtained the exponents 
\[
	a = b = \frac{1}{2n}. 
\] 
Lochak also discovered the remarkable phenomenon known as ``stability by resonance'', that if the initial condition is close to a $d$-resonance of low order, then one expects the stability exponents $a = b = \frac{1}{2(n-d)}$. By taking advantage of this fact, and that $1-$resonances divide the space, in \cite{BM11}, Bounemoura and Marco obtained larger stability exponent $a$ by allowing larger stability region (i.e. smaller $b$). The exponents obtained are 
\[
	b = (n-1) \sigma, \quad  a  =  \frac{1}{2(n-1)} - \sigma 
\]
where $\sigma>0$ can be arbitrarily small. 
The exponent $a$ can be taken to be $\frac{1}{2(n-1)}$ if one allows stability region of order $1$.

On the flip side, one is interested in the instability question known as Arnold diffusion. This research was started by the nominal work of Arnold (\cite{Arn64}), where he discovered the first mechanism for instability for nearly integrable Hamiltonian systems. Bessi (\cite{Bes96}, \cite{bes97}) proved that for $n =3, 4$,  there exists diffusion orbits $(\theta, I)(t)$, for which there exists $t>0$ such that   
\[
	\|I(t) - I(0)\| \ge C^{-1}, \quad |t| \le C \exp \left( C^{-1} \epsilon^{-\frac{1}{2(n-2)}} \right). 
\]
This result was then generalized to arbitrary $n \ge 5$ by the second author of this paper (\cite{Zha11}, see also related work in \cite{LM05}). The reason for the exponent $\frac{1}{2(n-2)}$ is due to restriction of Arnold's mechanism: the orbit constructed using Arnold's idea must always cross a double resonance, therefore the exponents obtained are the best allowed in that class. 

Up to now, there was still a gap between the best lower bound and upper bound of the stability exponent $a$: 
\[
	\frac{1}{2(n-1)}-\sigma \le a < \frac{1}{2(n-2)}. 
\]
In this paper, we close this gap by improving the stability exponents to 
\begin{equation}
  \label{eq:opt-exp}
  	b = \frac{n-2}{4} \sigma, \quad a = \frac{1}{2(n-2)} - \sigma. 
\end{equation}
Thus, the stability exponent $a$ can be arbitrary close to $\frac{1}{2(n-2)}$, and for Arnold diffusion, the exponent $\frac{1}{2(n-2)}$ is optimal. 

We obtain the improvements by separating the frequency space into two sets, one is close to resonances of order up to $|\log \epsilon|$, and the complement which is sufficiently non-resonant. In the non-resonant region we provide an improved stability result using first a normal form, then applying the Nekhoroshev's theory. In the resonant region, we apply an argument similar to the one in \cite{BM11}, to show that the fast diffusion orbit has to be close to a double resonance. 

The paper is structured as follows: in section~\ref{sec:main} we introduce notations and formulate the result. We also reduce the main theorem to two stability results, in the non-resonant and resonant regions. These results are proven in sections \ref{sec:non-res} and \ref{sec:stab-res}. 

\section{Formulation of the main result}
\label{sec:main}

 For $D \subset \R^n$ and  $r >0$ define:
\[
	V_r D = \{ I \in \C^n: \quad d(I, B) < r\} , \quad 
	U_r D = V_r D  \cap \R^n, 
\]
and 
\[
	V_{r,s}D = 
	\{ (\theta, I) \in \C^n \times (\C/2\pi \Z)^n: 
	\quad 
	d(I, B) < r, \quad |\Im(\theta)| < s. 
	\}
\]
where $d(x,y) = \max_i|x_i - y_i|$ is induced by the sup-norm in $\C^n$. consider the space $\cA_{r, s}(D)$ of real analytic functions  $\varphi(\theta, I)$ that is complex analytic on $V_{r,s}D$. The norm on this space is the sup-norm 
\[
	|\varphi|_{D, r, s} = \sup_{(\theta, I) \in V_{r,s}D}|\varphi(\theta, I)|. 
\]

Let $R, r_0, s_0, m, M >0$  be parameters let  $B(0, R) \subset \R^n$ be the ball of radius $R$, we assume the following conditions for $h$: 
\begin{itemize}
 \item $h\in \cA_{r_0,s_0}(B(0, R))$.
 \item $h$ is $l, m$-quasi-convex on $U_{r_0}B(0,R)$, namely, for all $I \in U_{r_0}B(0, R)$,   $\nabla h(I) \ne 0$, and 
\[
	\nabla^2 h(I) v \cdot v \ge m\|v\|^2, \quad \text{ if }
	  |v \cdot \nabla h(I)| \le l \|v\|^{{1}}.  
\]
\item $|\nabla h(I)|, |\nabla^2 h(I)| \le M$ for all $I \in U_{r_0}B(0, R)$.   
\end{itemize}
Let us denote $\cM = (R, r_0, s_0, l, m, M)$ the ensemble of parameters, and we reserve the notation $C = C(\cM)$ or $C_k = C_k(\cM)$ for unspecified positive constants depending only on $\cM$. The following is our main theorem. 

\begin{thm}\label{thm:main}
Under the standard assumptions on $h$, for any $0< \delta < \frac{1}{2n+4}$, there is $C = C(\cM) >1$, and $\epsilon_0 = \epsilon_0(\delta, \cM) >0$ such that if 
\[
	|f|_{D, r_0, s_0} < \epsilon \le \epsilon_0,
\]
then for all solutions $(\theta, I)(t)$ of $H$ with $I(0) \in B = B(0, R/2)$, we have 
\[
	|I(t) - I(0)| < C \epsilon^{\delta}, \text{ for }
	|t| < C^{-1} \exp\left( C^{-1} \epsilon^{-\frac{1-8\delta}{2n-4}} \right).  
\]
\end{thm}
\begin{rmk}
\eqref{eq:opt-exp} follows by taking $\sigma = \frac{8\delta}{2(n-2)}$. 
\end{rmk}

The theorem is proven by dividing the $I$-space into two regions: neighborhood of lower order $1$-resonance, and the complement. We produce a stability result on each region. 

Let $\Lambda \subset \Z^n$ be a submodule, the space of $\Lambda$ resonant frequencies is defined by 
\[
	R_\Lambda = \{ 
	\omega \in \R^n: \quad k \cdot \omega = 0
	\text{ for all } k \in \Lambda
	\}. 
\]
The associated resonance surface is 
\[
S_\Lambda = \{I \in \R^n: \quad \omega(I) \in R_\Lambda\}. 	
\]
We say that $\Lambda$ has rank $d$ if there is linearly independent $\{k_1, \cdots, k_d\} \subset \Z^n$ such that $\Lambda = \Span_\Z\{k_1, \cdots, k_d\}$.  In this case, we also write $R_\Lambda = R_{k_1, \cdots, k_d}$. $\Lambda$ is called maximal if it's not contained by a larger module of the same rank.   Following \pos, we say that $\Lambda$ is a $K$-module if is generated by $|k_i|\le K$, for all $i=1,\cdots,d$. 

Given a parameter $0 < \beta <1 $, we define
\begin{equation}
  \label{eq:para}
 	L = 12 s_0, \quad K(\epsilon) = - L \log \epsilon, \quad 
	r(\epsilon) = \beta^{-1} \epsilon^{\frac12}, \quad
	\alpha(\epsilon) = \beta^{-1} r(\epsilon) K(\epsilon), 
\end{equation}
and 
\[
	\cN(\epsilon) = 
	\{ 
	I \in B(0, R): \quad d(\omega(I), \bigcup_{0 < |k| \le K}R_k) < \alpha(\epsilon)
	\}, \quad 
	\cD(\epsilon) = B(0, R) \setminus \cN(\epsilon). 
\]

\begin{figure}[t]
	\centering
  \includegraphics[width=2.5in]{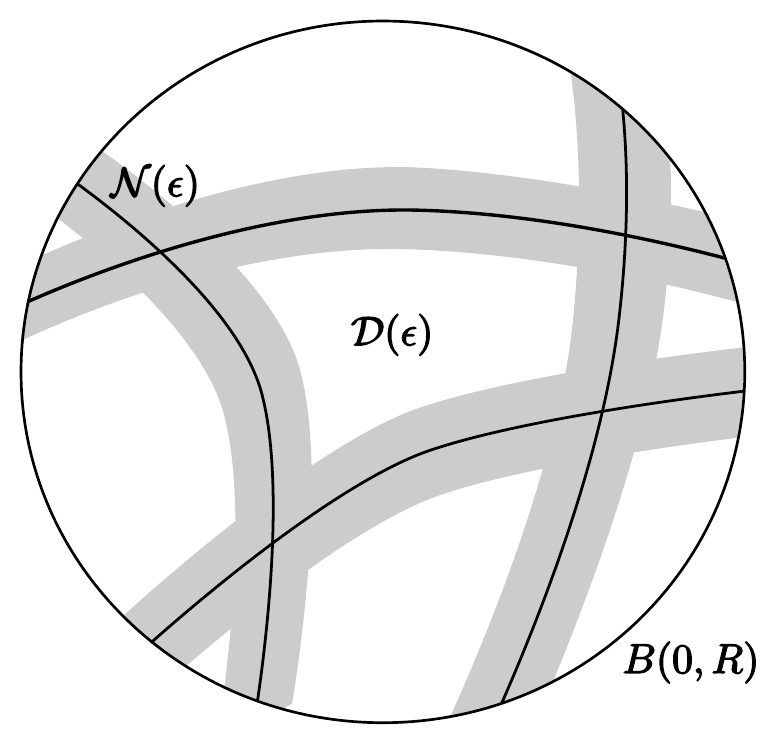}
  \caption{Resonant and non-resonant areas}
\end{figure}

Let $\Lambda \subset \Z^n$ be a maximal submodule, then a set  $D  \subset \R^n$ is called $\alpha, K$ non-resonant modulo $\Lambda$ if $|k \cdot \omega(I)| > \alpha$ for all $k \in \Z^n_K\setminus \Lambda$. $D$ is called $\alpha, K$ fully non-resonant if $\Lambda = \{0\}$. Then the set $\cN(\epsilon)$ is $\alpha(\epsilon)$ close to $K$-strong resonances, while the set $\cD(\epsilon)$ sufficiently non-resonant. 

The main observation is that orbits in the fully non-resonant region are much more stable than expected.

\begin{prop}\label{prop:non-res-stab}
Let $N = L/(6s_0)$. Under the our standing assumptions, there is $\epsilon_0 = \epsilon_0(\cM), \beta = \beta_0(\cM) >0$, $C_1 = C_1(\cM)>1$ such that for $\cD(\epsilon)$ defined using \eqref{eq:para}, if $|f|_{B(0, R), r_0, s_0} < \epsilon \le \epsilon_0$,  the following hold for $H = h + f$:

 {\it   Suppose $(\theta(t), I(t))$ be an orbit of $H$ starting with $I(0) \in \cD(\epsilon)$, then 
\[
	|I(t) - I(0)| \le C_1 \epsilon^{\frac12}, \quad 
	\text{ for }
	|t| \le C_1^{-1} \exp( C_1^{-1} \epsilon^{-\frac{N}{2n}}). 
\]
}
\end{prop}
\begin{rmk}
 choosing $L = 12s_0$ implies $N = 2$, and the stability time is $C_1 \exp( C_1^{-1} \epsilon^{-\frac{1}{n}})$, which is much longer than what's claimed in Theorem~\ref{thm:main}. The reason is the stability time is completely determined by what's happening in the resonant region. 
\end{rmk}

\begin{prop}\label{prop:1-res-stab}
For $0 < \delta < \frac{1}{2n+4}$, there exists $C_2 = C_2(\cM)>1$ and $\epsilon_0 = \epsilon_0(\cM, \delta) >0$ such that for 
\begin{equation}
  \label{eq:sta-time}
	T = C_2^{-1} \exp( C_2 ^{-1}\epsilon^{-\frac{1-8\delta}{2n-4}}),   
\end{equation}
and $0 < \epsilon \le \epsilon_0$,  the following hold: 

Let $(\theta(t), I(t))$, $t \in [0, T]$ be an orbit of $H$ with $I(0)\in B(0, R/2)$ and  $I(t) \in \cN(\epsilon)$ for all $t \in [0,T]$. Then 
\[
	|I(t) - I(0)| \le C_2 \epsilon^{\delta} 
	\text{ for all } 
	t \in [0, T]. 
\]
\end{prop}

\begin{proof}[Proof of Theorem~\ref{thm:main}]
Let $T$ be as in \eqref{eq:sta-time}, and assume that $\epsilon_0$ is small enough so that 
\[
	T \le C_1^{-1} \exp(C_1^{-1} \epsilon^{-\frac{N}{2n}}). 
\]
Consider any orbit $(\theta, I)(t)$, $t \in [0, T]$ such that $I(0) \in B(0, R/2)$. If there is $t_* \in [0, T]$ such that $I(t_*)\in \cD(\epsilon)$, then by Proposition~\ref{prop:non-res-stab}, 
\[
	|I(t) - I(t_*)| \le C \epsilon^{\frac12}, \quad t \in [t_*-T, t_* +T] \supset [0,T]. 
\]
Alternatively, $I(t)\in B(0, R/2)\cap \cN(\epsilon)$ for all of $[0, T]$, then Proposition~\ref{prop:1-res-stab} applies, and the theorem follows. 
\end{proof}

\section{Stability in the non-resonant region}
\label{sec:non-res}

Let $\Lambda \subset \Z^n$ be a maximal submodule, define the projection operator $T_K \varphi$ for $\varphi(\theta, I) = \sum_{k \in \Z^n} \varphi_k(I) e^{(k \cdot \theta)i}$ as follows: 
\[
	T_K \varphi = \sum_{|k| \le K}\varphi_k(I) e^{(k \cdot \theta)i}, \quad P_\Lambda \varphi = \sum_{k \in \Lambda}\varphi_k(I) e^{(k \cdot \theta)i}. 
\]
The function $\varphi$ is called resonant modulo $\Lambda$ if $P_\Lambda \varphi = \varphi$. 

We have the following resonant normal form lemma: 
\begin{lem}[\cite{Pos93}, Normal Form Lemma, page 192]\label{lem:res-normal-form}
Suppose $D \subset B(0,R)$ is $\alpha, K$-nonresonant modulo $\Lambda$, and $h$ satisfies the standing assumptions. There is $C_3 = C_3(\cM) >1$ such that, if $0 < r \le  r_0$, $0 < s \le s_0$ and $f \in \cA_{r, s}D$ satisfies
\[
  \|f\|_{D, r,s} \le \epsilon \le C_3^{-1} \frac{\alpha r}{K}, \quad
  r \le C_3^{-1} \frac{\alpha}{K},
\]
and $Ks \ge 6$, then there exists a real analytic coordinate change $\Phi: V_{r_1, s_1}D \to V_{r,s}D$ with $r_1 = r/2$, $s_1 = s/6$ such that $H \circ \Phi = h + g_1 + f_1$ with  
\[
	\|g_1 - g_0\|_{D, r_1, s_1} \le C_3 \frac{K}{\alpha r} \epsilon^2, \quad \|f_1\|_{D, r_1, s_1} \le e^{-Ks/6}\epsilon, 
\] 
where $g_0 = P_\Lambda T_K f$ and $P_\Lambda g = g$. Moreover, $\|\Pi_I \Phi - I\| \le C_3 \frac{K}{\alpha} \epsilon$ uniformly on $V_{r_1, s_1}D$, where $\Pi_I$ denote the projection $(\varphi, I)\mapsto I$. 
\end{lem}

We apply Lemma~\ref{lem:res-normal-form} to the fully non-resonant case $\Lambda = \{0\}$, then $g_0, g$ depends only on $I$.

\begin{cor}\label{cor:norm-sys}
Assume the standing assumptions for $h$, and let $\alpha(\epsilon), r(\epsilon), K(\epsilon)$ be chosen as in \eqref{eq:para}. 

 Write $r_1(\epsilon) = r(\epsilon)/2$, $r_2(\epsilon) = r(\epsilon)/4$ and $s_1 = s_0 /6$. Then there exists $\epsilon_0 = \epsilon_0(\cM), \beta_0(\cM)$such that if $\epsilon < \epsilon_0$ and $\beta < \beta_0$,  there exists 
\[
	\Phi: V_{r_1(\epsilon), s_1} \cD(\epsilon) \to V_{r(\epsilon), s_0} \cD(\epsilon), 
\]
such that $H \circ \Phi = h_1 + f_1$, with
\begin{enumerate}
 \item $h_1$ is $l/2, m/2$-quasi-convex on $U_{r_2(\epsilon)}\cD(\epsilon)$. 
 \item $|\nabla h_1(I)|, |\nabla^2 h_1(I)| \le 2M$ for all $I \in U_{r_2(\epsilon)}\cD(\epsilon)$. 
 \item $|f_1|_{\cD(\epsilon), r_1(\epsilon), s_1} \le \epsilon^{1+N}$. 
 \item $|\Pi_I \Phi - I| \le \epsilon^{\frac12}$. 
\end{enumerate}
\end{cor}
\begin{proof}
Let $D = \cD(\epsilon)$, $r = r(\epsilon)$, $\alpha = \alpha(\epsilon)$, $K = K(\epsilon)$, $\Lambda = \{0\}$. Then for  $\beta < C_3^{-1}$, 
\[
	C_3^{-1}\frac{\alpha r}{K} 
	= C_3^{-1} \beta^{-1} r^2 = C_3^{-1} \beta^{-3} \epsilon > \epsilon, \quad 
	C_3^{-1}\frac{\alpha}{K} = C_3^{-1} \beta^{-1} r > r.
\]
Therefore, Lemma~\ref{lem:res-normal-form} applies. It follows that $H \circ \Phi = h + g_1 + f_1$, with 
\[
	\|g_1 - g_0\|_{\cD, r_1, s_1} \le C_3 \beta r^{-2}\epsilon^2 =  C_3 \beta^3 \epsilon < \epsilon, \quad 
	\|f_1\|_{\cD, r_1, s_1} \le e^{-Ks/6}\epsilon = \epsilon^{N+1}. 
\]
Since $\Lambda$ is the trivial module, $g_1, g_0$ depends only on $I$, and $\|g_0\|_{\cD, r_1, s_1} \le \|f\|_{\cD, r, s_0} = \epsilon$.   Define 
\[
	h_1 = h + g_1, 
\]
using Cauchy estimates we have 
\[
	\|\nabla h_1 - \nabla h\|_{U_{r_2}\cD} \le (r_1/2)^{-1} \|g_1\|_{U_{r_1}\cD} \le 8 \beta \epsilon^{\frac12} < 8 \epsilon^{\frac12},  
\]
\[
	\|\nabla^2 h_1 - \nabla^2 h\|_{U_{r_2}\cD} \le (r_1/2)^{-2} \|g_1\|_{U_{r_1}\cD} \le 32 \beta^2.   
\]
Choose $\epsilon_0$, $\beta_0$  such that
\[
	8 \epsilon^{\frac12} < \min\{M, l/2\}, \quad 
	32 \beta^2 < \min\{m/2, M\}
\]
Then $\|\nabla h_1\|, \|\nabla^2 h_1\| \le 2M$ on $U_{r_2}\cD$. To prove quasi-convexity, note that one of the following holds for all $\|v\| =1$:
\[
	\|\nabla^2 h(I) v\cdot v\| \ge m,  \text{ or } \|\nabla h(I) \cdot v\| > l.
\]
Our estimates imply one of the following always hold:
\[
	\|\nabla^2 h_1(I) v \cdot v\|\ge m/2, \text{ or } \|\nabla h_1(I) \cdot v\| > l/2,
\]
implying  $l/2, m/2$-semi-concavity. 
\end{proof}

We then apply the following \emph{global} stability theorem, which we apply to the normal form system. It's important to note that $h_1$ does \emph{not} satisfy our standing assumption, and special care needs to paid to which parameters the constants depends on. 
\begin{thm}[\cite{Pos93}, Theorem 1] \label{thm:global-stab}
Suppose $H = h_1 + f_1 \in \cA_{r, s}D$, $h_1$ is $l,m-$quasi-convex and 
\[
\|\nabla^2 h_1(I)\|_{D, r,s} \le M	. 
\]
There is $C_4>1$ depending on $s, l,  m, M$ such that the following hold. For $r \le s$,  let $f_1 \in \cA_{r,s}(D)$ satisfy
\[
	\|f_1\|_{D, r, s} \le \epsilon \le \epsilon_0 = C_4^{-1} r^2.
\]
Then for every orbit of $H$ with $(\theta(0), I(0))\in \T^n \times D$, one has 
\[
	\|I(t) - I_0\| \le C_4 r \left( \frac{\epsilon}{\epsilon_0} \right)^{\frac{1}{2n}}, \quad \text{ for } 
	|t| \le C_4^{-1} \exp \left( C_4^{-1} \left( \frac{\epsilon_0}{\epsilon}  \right)^{\frac{1}{2n}} \right). 
\]
\end{thm}
\vspace{10pt}

\begin{proof}[Proof of Proposition~\ref{prop:non-res-stab}]
Let $r_3 = r_3(\epsilon) = r_2(\epsilon)/2$. 

Let $(\theta, I)(t)$ be an orbit of $H$ with $I(0) \in \cD(\epsilon)$, then $(\theta', I')(t) = \Phi^{-1}(\theta, I)(t)$ is an orbit of $H \circ \Phi$ as long as $I'(t) \in U_{r_1, s_1}\cD(\epsilon)$. Note that according to item 4 of Corollary~\ref{cor:norm-sys}, $|I'(0) - I(0)| < \epsilon^{\frac12} < r_3 = \beta^{-1} \epsilon^{\frac12} / 8$ , so $I'(0) \in U_{r_3}\cD$. 

Consider 
\[
	H \circ \Phi = h_1 + f_1, \quad D = U_{r_3}\cD(\epsilon), 
\]
then Theorem~\ref{thm:global-stab} applies with parameters $r_3, s_1$ since
\[
	\epsilon_1 = \epsilon^{N+1} \le \bar{\epsilon}_0 = C_4^{-1} r_3^2 = (64C_4)^{-1} r^2 = (64C_4) \beta^{-2} \epsilon. 
\]
Therefore 
\[
	\|I'(t) - I'(0)\| \le C_4 r_3 \left( \frac{\epsilon_1}{\bar{\epsilon}_0} \right)^{\frac{1}{2n}} \le  C_4 \beta^{-1} \epsilon^{\frac12} \left( \frac{\epsilon^{1+N}}{\epsilon} \right)^{\frac{1}{2n}} \le C_4 \beta^{-1} \epsilon^{\frac12 + \frac{N}{2n}},
\]
since $\bar\epsilon_0 > \epsilon$; for the time interval
\[
	|t| \le 
	C_4^{-1} \exp \left( C_4  \epsilon^{\frac{N}{2n}} \right) \le 
	C_4^{-1} \exp \left( C_4^{-1} \left( \frac{\bar{\epsilon}_0}{\epsilon_1} \right)^{\frac{1}{2n}} \right),
\]
which includes the time interval 
\[
	|t| \le C_4^{-1} \exp \left( C_4^{-1}  \epsilon^{-\frac{N}{2n}}\right). 
\]
Using $I'(t) \in U_{r_1, s_1}\cD(\epsilon)$, and $|I(t) - I'(t)| < \epsilon^{\frac12}$ we obtain our proposition. 
\end{proof}

\section{Stability near strong $1$-resonances}
\label{sec:stab-res}

Suppose $\Lambda \subset \Z^n$ is a maximal submodule, and let $k_1, \cdots, k_d \in \Z^n$ be linearly independent and generates $\Lambda$ over $\Z$. The volume $|\Lambda|$ of $\Lambda$ is defined as 
\[
	|\Lambda|^2 = \det \bmat{k_1^T \\ \vdots \\ k_d^T} 
	\bmat{k_1 & \cdots & k_d}. 
\]
This definition is independent of the basis $k_1, \cdots, k_d$. 
$\Lambda$ is called a $K$-lattice if $|k_1|, \cdots, |k_d| \le K$. 
\begin{thm}[\cite{Pos93}, Theorem 3]\label{thm:res-stability}
Suppose $h$ satisfies the standing assumption, and consider a $K_\Lambda-$lattice $\Lambda$ of dimension $d$. Then there exist $C_5 = C_5(\cM)>1$ such that if 
\[
	|f|_{B, r, s} \le  \epsilon \le \frac{\epsilon_\Lambda}{K_\Lambda^{2(n-d)}}, \quad 
	\epsilon_\Lambda = C_5^{-1} |\Lambda|^{-2},
\]
where $|\Lambda|$ is the volume of $\Lambda$. Then for every orbit $(\theta, I)(t)$ such that 
\[
	d(\omega(I(0)), R_\Lambda) < C_5^{-1}\sqrt{\epsilon}, 
\]
 one has 
\[
	\|I(t) - I_0\| \le C_5r\left( \frac{\epsilon}{\epsilon_\Lambda} \right)^{\frac{1}{2(n-d)}}, 
	\text{ for }
	|t| \le C_5^{-1} \exp \left( C_5^{-1}\left( \frac{\epsilon_\Lambda}{\epsilon} \right)^{\frac{1}{2(n-d)}} \right). 
\]
\end{thm}

The stability in the resonant area follows by two steps. First, by geometric consideration, we show that any orbit which drifts a large enough distance, in the neighborhood of strong $1$-resonance must be close to a $2$-resonance $R_{k_1, k_2}$ with estimates on $|k_1|, |k_2|$. We then apply Theorem~\ref{thm:res-stability}.

\begin{lem}\label{lem:cross-res}
Let $(\theta, I)(t)$ be an orbit of $H$ with $\|I(T) - I(0)\|> \epsilon^\delta$, and $I(t) \in \cN(\epsilon)$ for all $t \in [0, T]$. Then there exists $C_6 = C_6(\cM)$, $t_* \in [0,T]$ and 
\[
	k_1, k_2 \in \Z^n \setminus\{0\}, \quad |k_1| \le K, \quad 
	|k_2| \le C_6 \epsilon^{-\delta},
\]
such that 
\[
	d(\omega(I(t_*)), R_{k_1, k_2}) < C_6 \beta^{-2} \epsilon^{\frac12 - \delta} K^2(\epsilon). 
\]
\end{lem}

\begin{figure}[t]
  \centering
  \includegraphics[width=2.5in]{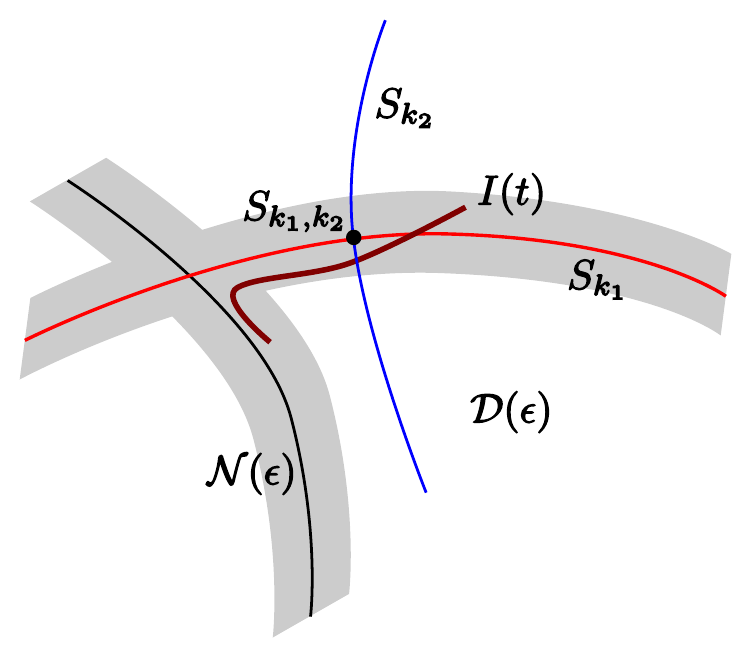}
  \caption{Any curve of sufficient length in $\cN(\epsilon)$ must pass close to a double resonance}
\end{figure}

First we have the following lemma, which is a modified version of Lemma 3.4 from \cite{BM11}. 
\begin{lem}\label{lem:rational}
Let $I \subset [-1, 1]$ be a closed interval of length $l > 0$. Suppose $0 < K^2 < 2 l^{-1}$, then there is an irreducible rational number $p/q \in I \cap \Q$ such that 
\[
	K < q < 3 l^{-1}. 
\]
\end{lem}
\begin{proof}
Let $Q = \floor{3l^{-1}} > 2 l^{-1}$, then there is $m \in \Z$ such that $\frac{m}{Q}, \frac{m+1}{Q} \in I$. We now show at least one of them satisfies the conclusion of the lemma. Indeed, if $\frac{p_1}{q_1}, \frac{p_2}{q_2} \in \Q$ are distinct and $|q_1|, |q_2| \le K$, then $\left| \frac{p_1}{q_1} -  \frac{p_2}{q_2} \right| \ge \frac{1}{|q_1q_2|} > K^{-2} > Q^{-1}$, therefore at most one of $\frac{m}{Q}$ and $\frac{m+1}{Q}$ can have denominator bounded by $K$ when reduced. 
\end{proof}

\begin{proof}[Proof of Lemma~\ref{lem:cross-res}]
	The proof is inspired by Lemma~3.3 of \cite{BM11}. Consider the map 
	\[
	\Psi_h: (I, \lambda) \mapsto (h(I), \lambda \omega(I)),
	\]
	then $\Psi_h$ is a local diffeomorphism. Therefore, there exists $\rho_0, C >0$ depending on $\cM$ such that 
	\[
	|\Psi_h(I_1, \lambda_1) - \Psi_h(I_2, \lambda_2)| \ge C^{-1} |(I_1 - I_2, \lambda_1 - \lambda_2)|, \quad 
	\text{ if } |(I_1- I_2, \lambda_1 - \lambda_2)| < \rho_0. 
	\]
	Suppose $\epsilon_0$ is small enough that $\epsilon_0^\delta < \rho_0$. Write $\omega(t) = \omega(I(t))$ and  $t_0$ be the first time the curve $(I(t), |\omega(t)|^{-1})$ leaves the $\epsilon^\delta$ neighborhood of $(I(0), |\omega(0)|^{-1})$, with $0 < \epsilon < \epsilon_0$. Then the above observation implies 
	\[
	\left|\left( h(I(t_0))- h(I(0)),\,   \frac{\omega(t_0)}{|\omega(t_0)|} - \frac{\omega(0)}{|\omega(0)|} \right) \right| > C^{-1} \epsilon^\delta. 
	\]
	Since energy conservation implies $|h(I(t_0))- h(I(0))| < 2\epsilon$, we obtain 
	\[
	\left| \frac{\omega(t_0)}{|\omega(t_0)|} - \frac{\omega(0)}{|\omega(0)|} \right| > C^{-1}\epsilon^\delta, 
	\]
	which implies for some index $i \in \{1, \cdots, n\}$, the interval $\{ \omega_i(t)/|\omega(t)|: t \in [0, t_0]\} \subset [-1, 1]$ has  length at least $C^{-1}\epsilon^\delta$. Then according to Lemma~\ref{lem:rational}, there exists irreducible $p/q \in \Q$ with $K < |q| < 3 C\epsilon^{-\delta}$ and $t_* \in [0, t_0]$, such that $\omega_i(t_*)/|\omega(t)| = p/q$. Let $j \in \{1, \cdots, n\}$ such that $\omega_j(t_*)/|\omega(t_*)| = 1$, (which exists since $|\omega|=\sup_j|\omega_j|$), it follows that for $k_2 = q e_i + p e_j \in \Z^n$, where $e_i$ denotes the coordinate vectors, we have 
	\[
	k_2 \cdot \omega(t_*) = 0, \quad |k_2| \le 6 C \epsilon^{-\delta}. 
	\]
	Moreover, since $|q| > K$ and $k_2$ is irreducible, $k_2$ cannot be generated by any vector with $|k| \le K$, therefore $\{k_1, k_2\}$ is linearly independent. 

	Since $I(t_*) \in \cN(\epsilon)$, there exists $0 < |k_1| \le K$ such that  $\dist(\omega(t_*), R_{k_1}) < \alpha(\epsilon)$. Let $\bar\omega$ be the projection of $\omega(t_*)$ to the hyperplane $R_{k_1} \cap R_{k_2}$, we first note
	\[
	\sin \angle(R_{k_1}, R_{k_2}) = \sin \angle (k_1, k_2) = \frac{\sqrt{\|k_1\|^2 \|k_2\|^2 - (k_1 \cdot k_2)^2}}{\|k_1\| \|k_2\|} \ge \frac{1}{|k_1| |k_2|}, 
	\]
	 then 
	\[
	|\omega(t_*) - \bar\omega| \le \frac{d(\omega(t_*), R_{k_1})}{\sin \angle (R_{k_1}, R_{k_1})} 
	\le \alpha(\epsilon)|k_1| |k_2| \le 6C \alpha(\epsilon) \epsilon^{-\delta} K(\epsilon)
	\]
	and the lemma follows from taking $C_6 = 6C$, and plugging in $\alpha(\epsilon) = \beta^{-1} r(\epsilon)K(\epsilon)$ and $r(\epsilon) = \beta^{-1} \epsilon^{\frac12}$. 
\end{proof}

According to our definition, $R_{k_1, k_2}$ is generated by the module $\Span_\Z\{k_1, k_2\}$, which is not necessarily maximal. 
In order to apply Theorem~\ref{thm:res-stability}, we need the following lemma. 
\begin{lem}\label{lem:Lb-bound}
Suppose $\Lambda$ is the  maximal module containing $k_1, k_2$, namely
\[
\Lambda = \Span_\R\{k_1, k_2\} \cap \Z^n	
\]
where $\{k_1, k_2\}$ is linearly independent. Then $\Lambda$ is a $|k_1| + |k_2|$-lattice, and $|\Lambda| \le |k_1| |k_2|$. 
\end{lem}
\begin{proof}
The lemma is non-trivial because $k_1, k_2$ does not necessary generate $\Lambda$ over $\Z$. 

We first derive a relation for arbitrary number of generators. Suppose $\Lambda$ is the maximal module containing $k_1, \cdots, k_d$, let $k_1', \cdots, k_d'$ generate $\Lambda$ over $\Z$. Let $A$ be the matrix with columns $k_1, \cdots, k_d$, and $B$ with columns $k_1', \cdots, k_d'$. Then there exist invertible $d\times d$ integer matrix $G$ such that 
\[
	A = B G. 
\]
Then 
\[
	\det (A^T A) = \det (G^T B^T B G) = \det (G^T) \det (B^T B) \det (G) \ge \det(B^T B) = |\Lambda|^2. 
\]
We now go to the case $d=2$, we  have $|\Lambda|^2 = \|k_1\|^2 \|k_2\|^2 - (k_1 \cdot k_2)^2 \le \|k_1\|^2 \|k_2\|^2$, and the estimate follows from $\|k\| < |k|$. 

To prove $\Lambda$ is a $|k_1| + |k_2|$ lattice, we claim there exists $k_1', k_2'$ generating $\Lambda$ with $|k_1'|, |k_2'| \le |k_1| + |k_2|$. The argument presented here is based on the more general argument in \cite{Sie}, Theorem 18. Define 
\[
	s_2 = \min\{t_2>0: \, \R k_1 + t_2 k_2 \cap \Lambda \ne \emptyset\},\quad
	 s_1 = \min\{t_1 \ge 0: \, t_1 k_1 +  s_2 k_2 \in \Lambda \},
\]
and $k_2' = s_1 k_1 + s_2 k_2$. We now show $k_1', k_2'$ generates $\Lambda$ over $\Z$. For any $k \in \Lambda$, there exists $t_1, t_2 \in \R$ such that $k = t_1 k_1' + t_2 k_2'$. Assume that $t_2 \notin \Z$, then there exists $n \in \Z$ such that $0 < a = t_2 + n < 1$. We have 
\[
	k + n k_2' = t_1 k_1' + a k_2' = (t_1+ a s_1) k_1 + a s_2 k_2 \in \Lambda.
\]
Since $0 < a s_2 < s_2$, this contracts with the minimality of $s_2$. As a result $t_2 \in \Z$. We can show $t_1 \in \Z$ by the same argument. Since $0 \le s_1 < 1$ and $0 < s_2 \le 1$ by definition, we know $|k_2'| < |k_1| + |k_2|$. 
\end{proof}

\vspace{10pt}
\begin{proof}[Proof of Proposition~\ref{prop:1-res-stab}]
Suppose $(\theta, I)(t)$, $t \in [0, T]$ is an orbit satisfying $I(t) \in \cN(\epsilon)$ for all $t$. Arguing by contradiction, suppose  $|I(t) - I(0)| > \epsilon^\delta$ for some $t \in [0, T]$. We apply Lemma~\ref{lem:cross-res}, to obtain that there exists $t_* \in [0, T]$, and $|k_1| \le K(\epsilon)$, $|k_2| \le C \epsilon^{-\delta}$, such that 
\[
	d(\omega(t_*), R_{k_1, k_2}) < C_6 \beta^{-2} \epsilon^{\frac12 - \delta} K^2(\epsilon).
\]

We will pick $\epsilon_0$ depending on $\delta$ such that for all $\epsilon< \epsilon_0$, we have 
\[
	K(\epsilon) = - L\log\epsilon \le \epsilon^{-\delta}. 
\]
Let $\Lambda = \Span_\R \{ k_1, k_2\} \cap \Z^n$ be the maximal lattice generated by $k_1, k_2$. According to Lemma~\ref{lem:Lb-bound}, 
\[
	K_\Lambda \le K(\epsilon) + C_6\epsilon^{-\delta} \le 2C_6 \epsilon^{-\delta}, \quad 
	1 \le |\Lambda| \le C_6 K(\epsilon) \epsilon^{-\delta} \le C_6 \epsilon^{- 2\delta} . 
\]

We attempt to apply Theorem~\ref{thm:res-stability} near the resonance $R_\Lambda$. Set 
\[
	\epsilon_2 = C_5 \left( C_6 \beta^{-2} \epsilon^{\frac12 -\delta} K^2 \right) = C_5 C_6^2 \beta^{-4} \epsilon^{1-2\delta} K^4 \le C_5 C_6^2 \beta^{-4} \epsilon^{1-6\delta}. 
\]
Then $d(\omega(t_*), R_{k_1, k_2}) < C_5^{-1} \sqrt{\epsilon_2}$. 
Then 
\[
	\epsilon_\Lambda = C_5^{-1} |\Lambda|^{-2}   \ge (C_5 C_6)^{-1} \epsilon^{2\delta}, \quad 
	\frac{\epsilon_\Lambda}{K_\Lambda^{2(n-2)}} \ge \frac{(C_5C_6)^{-1}\epsilon^{2\delta}}{(2C_6)^{2n-2}\epsilon^{-\delta(2n-4)}} 
	\ge C_7^{-1} \epsilon^{(2n-2)\delta}
\]
for $C_7 = (C_5C_6)(2C_6)^{2n-2}$. Then if $(2n+4)\delta < 1$ and $\epsilon$ small enough depending on $C_5, C_6, C_7, \beta$ and $\delta$, we have 
\[
	\epsilon_2 \le C_5 C_6^2 \beta^{-4} \epsilon^{1-6\delta} <  C_7^{-1} \epsilon^{(2n-2)\delta} \le  \frac{\epsilon_\Lambda}{K_\Lambda^{2(n-2)}},
\]
and 
\[
	\epsilon_2 /\epsilon_\Lambda \le C_5 C_6^2 \beta^{-4} \epsilon^{1-8\delta}. 
\]
 Theorem~\ref{thm:res-stability} applies and we obtain 
\[
	\|I(t) - I(t_*)\| \le C_5 \left( \frac{\epsilon_2}{\epsilon_\Lambda} \right)^{\frac{1}{2n-4}}
	\le  C_5 (C_5 C_6^2)^{\frac{1}{2n-4}} \left( \epsilon^{1-8\delta} \right)^{\frac{1}{2n-4}} 
	\le C_2 \epsilon^{\frac{1-8\delta}{2n-4}},
\]
for $C_2 = C_5 (C_5 C_6^2)^{\frac{1}{2n-4}}$, as long as 
\[
	|t - t_*| \le C_5^{-1} \exp \left(  C_5^{-1} \left( \frac{\epsilon_\Lambda}{\epsilon_2} \right)^{\frac{1}{2n-4}} \right)   
\]
which is implied by 
\[
	|t - t_*| \le C_2^{-1} \exp \left( C_2^{-1} \epsilon^{-\frac{1-8\delta}{2n-4}} \right). 
\]
This in particular would imply $|I(T)-I(0)| \le |I(T)-I(t_*)| + |I(0) - I(t_*)| \le 2C_2 \epsilon^{\frac{1-8\delta}{2n-4}}$. Since $(2n+4)\delta <1$, then $\frac{1-8\delta}{2n-4} >  \delta$, therefore for $\epsilon$ small enough $|I(T)-I(0)| \le 2C_2 \epsilon^{\frac{1-8\delta}{2n-4}} < \epsilon^\delta$ which is a contradiction.
\end{proof}
\subsection*{Acknowledgments}

The authors thank Abed Bounemoura for valuable discussions. K. Zhang was supported by the NSERC Discovery grant, reference number 436169-2013.

\bibliographystyle{plain}
\bibliography{diffusion}

\end{document}